\documentclass[12pt]{article}
\usepackage{e-jc}

\usepackage[english]{babel}
\usepackage{amsthm,amsmath,amssymb}
\usepackage{graphicx}
\usepackage[small,bf,hang]{caption} 

\usepackage{algorithm}                   
\usepackage{algorithmic}

\usepackage{epsf} 

\usepackage{hyperref}
\usepackage{enumerate} 

\usepackage{multirow}

\usepackage{color}

\newtheorem{theorem}{Theorem}
\theoremstyle{plain}
\newtheorem{conjecture}[theorem]{Conjecture}

\newtheorem{proposition}[theorem]{Proposition}
\newtheorem{corollary}[theorem]{Corollary}
\newtheorem{claim}[theorem]{Claim}

\theoremstyle{definition}

\newtheorem{question}[theorem]{Question}

\graphicspath{{figures/}}

\newenvironment{dedication}
{\begin{quotation}\begin{center}\begin{em}}
{\end{em}\end{center}\end{quotation}}

\title{\bf On minimal triangle-free 6-chromatic graphs}

\author{Jan Goedgebeur\thanks{Supported by a Postdoctoral Fellowship of the Research Foundation Flanders (FWO).}\\
\small Department of Applied Mathematics, Computer Science and Statistics\\[-0.8ex]
\small Ghent University\\[-0.8ex]
\small Krijgslaan 281-S9,\\[-0.8ex]
\small 9000 Ghent, Belgium\\[-0.8ex]
\small\tt jan.goedgebeur@ugent.be\\
\\
}

\date{\dateline{XX}{XX}}

\begin{document}

\maketitle
\begin{dedication}
\vspace*{-2em}
In loving memory of Ella.
\end{dedication}

\begin{abstract}
A graph with chromatic number $k$ is called \textit{$k$-chromatic}. Using computational methods, we show that the smallest triangle-free 6-chromatic graphs have at least 32 and at most 40 vertices. 

We also determine the complete set of all triangle-free 5-chromatic graphs up to 24 vertices. This implies that Reed's conjecture holds for triangle-free graphs up to at least this order. We also establish that the smallest regular triangle-free 5-chromatic graphs have 24 vertices.
Finally, we show that the smallest 5-chromatic graphs of girth at least 5 have at least 29 vertices and that the smallest 4-chromatic graphs of girth at least 6 have at least 25 vertices.

\end{abstract}


\section{Introduction}
\label{section:intro}

Throughout this paper, all graphs are simple and undirected. The \textit{chromatic number} of a graph $G$, denoted by $\chi(G)$, is the minimum number of colours required to colour the vertices of $G$ such that no two adjacent vertices have the same colour. A graph $G$ with $\chi(G) = k$ is called \textit{$k$-chromatic}. Let $\delta(G)$ and $\Delta(G)$ denote the minimum and maximum degree of $G$, respectively, or just $\delta$ and $\Delta$ if $G$ is clear from the context. The \textit{girth} of a graph is the length of its shortest cycle.

A triangle-free $k$-chromatic graph of order $n$ is called a $(k,n)$-graph, a $(k,n,d)$-graph a $(k,n)$-graph with $\Delta = d$ and a $(k,n, \leq d)$-graph a $(k,n)$-graph with $\Delta \leq d$. Finally, $n(k)$ is defined as the number of vertices of the smallest triangle-free $k$-chromatic graph. Note that $n(k)$ is in fact equal to the value of the vertex Folkman number $F_v(2^{k-1};3)$, see~\cite{xu2016chromatic}.

In~\cite{M55} Mycielski presented a construction which shows that there are triangle-free graphs with arbitrarily large chromatic number, i.e.\ $n(k)$ is well-defined. When the Mycielski construction is applied to a triangle-free $k$-chromatic graph on $n$ vertices, it yields a triangle-free $(k+1)$-chromatic graph on $2n+1$ vertices, hence $n(k+1) \leq 2n(k) + 1$. 

Chv{\'a}tal~\cite{C74} showed that the Gr\"otzsch graph (see Figure~\ref{fig:grotzsch}) -- which is obtained by applying the Mycielski construction to the 5-cycle -- is the smallest triangle-free 4-chromatic graph, so $n(4)=11$.

\begin{figure}[h!t]
	\centering
	\includegraphics[width=0.4\textwidth]{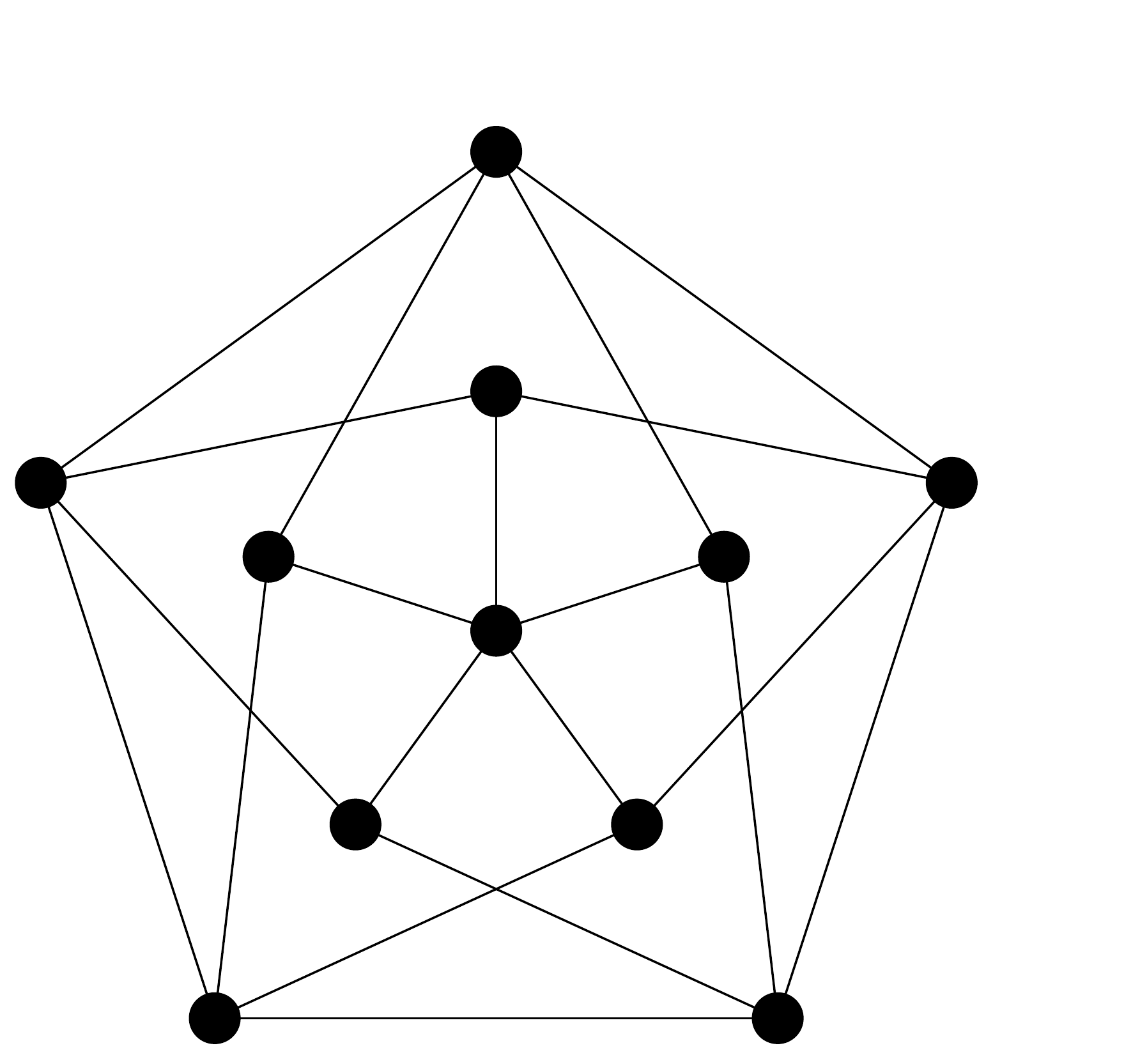}
	\caption{The Gr\"otzsch graph.}
	\label{fig:grotzsch}
\end{figure}


In~\cite{toft1988} Toft asked for the value of $n(5)$. Using a computer search Grinstead, Katinsky and Van Stone~\cite{grinstead1989minimal} showed that $21 \leq n(5) \leq 22$. In~\cite{jensen1995small}, also using a computer search, Jensen and Royle showed that $n(5)=22$ and that there are at least 80 $(5,22)$-graphs.

Later, Jensen and Toft asked~\cite{jensen1995graph} for the value of $n(6)$. Note that applying the Mycielski construction to the smallest triangle-free 5-chromatic graphs found by Jensen and Royle gives an upper bound of $n(6) \leq 45$. This was the best known upper bound until Droogendijk~\cite{droogendijk2015} presented a $(6,44)$-graph in 2015. Later he reported that he also found such graphs on 43 vertices, but did not provide any adjacency lists or details.

Using the technique from Section~\ref{subsect:maxdeg}, the fact that $n(5)=22$ and some maximum degree properties imply that $29 \leq n(6)$.

This paper is organised as follows. In Section~\ref{section:lower_bound} we present two methods for constructing triangle-free $k$-chromatic graphs, that is: the maximum triangle-free method in Section~\ref{subsect:mtf} and the maximum degree extension method in Section~\ref{subsect:maxdeg}. In Section~\ref{subsect:results} we describe how using a combination of these methods allowed us to improve the lower bound for $n(6)$ to 32. 
We also show that if a $(6,32)$-graph exists, it must have maximum degree 6, and that if a $(6,33)$-graph exists, it must have maximum degree 6 or 7.

We determine the complete set of all $(5,n)$-graphs up to $n=24$. This implies that Reed's $\omega$, $\Delta$, and $\chi$ conjecture~\cite{reed1998omega} holds for triangle-free graphs up to at least 24 vertices, see Section~\ref{subsect:results} for details. Next to that, we also determine the smallest regular triangle-free 5-chromatic graphs.

In~\cite{jensen1995graph} Jensen and Toft also asked for the order of the smallest 4-chromatic graph of girth at least 5. 
The Brinkmann graph, which is the smallest 4-regular 4-chromatic graph of girth 5~\cite{brinkmann1997smallest}, gives an upper bound of 21 vertices. Recently, Royle~\cite{royle2015} showed that the smallest 4-chromatic graphs of girth at least 5 have 21 vertices and that there are exactly 18 such graphs. Next to the Brinkmann graph, this set also contains the smallest 4-chromatic $P_{12}$-free graph of girth 5 found by Schaudt and the author~\cite{GS2017}.

In Section~\ref{subsect:results} we determine the complete set of all 4-chromatic graphs of girth 5 on 22 vertices and show that the smallest 4-chromatic graph of girth at least 6 has at least 25 vertices.
By adapting the technique from Section~\ref{subsect:maxdeg} for graphs of girth at least 5, we prove that the smallest 5-chromatic graph of girth at least 5 has at least 29 vertices.

In Section~\ref{section:upper_bound} we describe how, using a heuristic method, we constructed more than 750\,000 $(6,40)$-graphs. In~\cite{jensen1995graph} Jensen and Toft also asked the following question.

\begin{question}[Jensen and Toft~\cite{jensen1995graph}]
Is $n(k+1) \leq 2n(k)$ for all $k\geq 4$?
\end{question}

Recall that $n(5)=22$, so our $(6,40)$-graphs provide a positive answer to this question for $k=5$.

Our main result is:

\begin{theorem}\label{thm:bounds_6chrom}
$32 \leq n(6) \leq 40$.
\end{theorem}

Applying the Mycielski construction to a $(6,40)$-graph yields a $(7,81)$-graph and by applying the technique from Section~\ref{subsect:maxdeg} using the lower bound from Theorem~\ref{thm:bounds_6chrom} we obtain:

\begin{corollary}\label{cor:bounds_7chrom}
$41 \leq n(7) \leq 81$.
\end{corollary}


\section{Improving the lower bound}
\label{section:lower_bound}
 

\subsection{Maximum triangle-free method}
\label{subsect:mtf}

A \textit{maximal} triangle-free graph (in short, an \textit{mtf graph}) is a triangle-free graph such that the insertion of any new edge forms a triangle. For graphs with at least 3 vertices this is equivalent to being triangle-free and having diameter 2. Note that it is sufficient to restrict our algorithm to the generation of mtf $k$-chromatic graphs since there exists a triangle-free $k$-chromatic graph of order $n$ if and only if there exists an mtf $k$-chromatic graph order $n$ (under the assumption that there are no triangle-free graphs of order $n$ with chromatic number larger than $k$).
 Furthermore, the complete set of $(k,n)$-graphs can be obtained from the the complete set of mtf $(k,n)$-graphs by recursively removing edges in all possible ways, as long as the graphs stay $k$-chromatic.

Together with Brinkmann and Schlage-Puchta~\cite{triangleramsey-site,BGSP12} we designed an efficient algorithm for the generation of all non-isomorphic mtf graphs of a given order. We now used this generator and then applied a filter to its output to determine the $k$-chromatic graphs. Using this method we determined the complete set of all $(5,n)$-graphs up to $n=24$, see Section~\ref{subsect:results} for details. These sets of graphs are vital since they serve as input graphs for the maximum degree extension method in Section~\ref{subsect:maxdeg}.

It is currently computationally infeasible to generate complete sets of triangle-free $k$-chromatic graphs with more than 24 vertices using our mtf method. 

\subsection{Maximum degree extension method}
\label{subsect:maxdeg}


Let $N(v)$ be the neighbourhood of a vertex $v$. Given an $S \subseteq V(G)$, $G \setminus S$ denotes the subgraph of $G$ induced by $V(G) \setminus S$.

It is straightforward to see that the following holds.

\begin{proposition}
Let $G$ be an mtf $(k,n,d)$-graph and let $v$ be a vertex of maximum degree. Then $G' = G \setminus (N(v) \cup \{v\})$ is a $(k-1,n-d-1,\leq d-1)$-graph or a $(k,n-d-1,\leq d-1)$-graph.
\end{proposition}

Note that the neighbours of $v$ in $G$ are connected to maximal independent sets in $G'$. Moreover, for all of our computations it will be known that $(k,n-d-1)$-graphs do not exist so we can restrict ourselves to $(k-1,n-d-1,\leq d-1)$-graphs. 

Our algorithm is similar to the method used by Jensen and Royle in~\cite{jensen1995small} and basically it works as follows. To obtain the complete set of all $(k,n)$-graphs, we first construct the complete sets of all mtf $(k,n,d)$-graphs for every feasible maximum degree $d$. This is done by starting from every $(k-1,n-d-1,\leq d-1)$-graph $H$, adding a vertex $v$ with $d$ neighbours $n_1,...,n_d$ and connecting them to maximal independent sets of $H$ in all possible ways (without increasing the maximum degree) -- see Figure~\ref{fig:gluing}, and testing if the resulting graphs are mtf and $k$-chromatic.

The set of all $(k,n)$-graphs is then obtained by recursively removing edges in all possible ways from these mtf $(k,n)$-graphs, as long as the graphs stay $k$-chromatic. The pseudocode is given in Algorithm~\ref{algo:init} and Algorithm~\ref{algo:construct}.

\begin{figure}[h!t]
	\centering
	\resizebox{0.45\textwidth}{!}{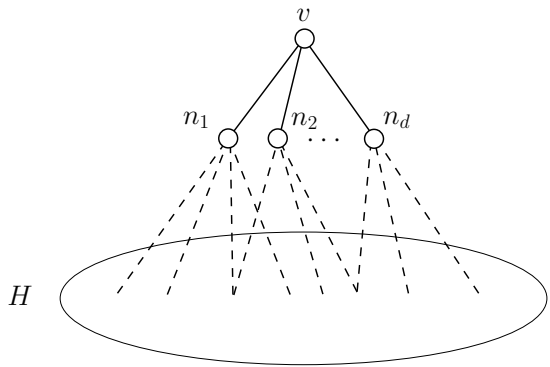}
	\caption{The operation of gluing the neighbours $n_1,...,n_d$ of $v$ to maximal independent sets of $H$ in all possible ways.} 
	\label{fig:gluing}
\end{figure}

\begin{algorithm}[h]
\caption{Generate all mtf $(k,n,d)$-graphs}
\label{algo:init}
  \begin{algorithmic}[1]
  \FOR{every $(k-1,n-d-1,\leq d-1)$-graph $H$}
  		\STATE Compute all maximal independent sets of orders $1,...,d-1$ in $H$.
  		\STATE \verb|Connect_indep_sets|($H, d, d, 0, 0$) \ // i.e.\ perform Algorithm~\ref{algo:construct}
  \ENDFOR
  \end{algorithmic}
\end{algorithm}

\begin{algorithm}[H]
\caption{{\tt Connect\_indep\_sets}(Graph $H$, int $d$, int set\_order, int set\_index, \\ int num\_assigned)
\\ // set\_order = \# neighbours of $n_i$ in $H$
\\ // set\_index = index of the next independent set which could be connected to $n_i$
\\ // num\_assigned = \# $n_i$'s which were already assigned to an independent set in $H$}
\label{algo:construct}
  \begin{algorithmic}
	\IF{num\_assigned == $d$}  
  		\STATE expand $H$ to $G$ \ // as in Figure~\ref{fig:gluing}
  		\IF{$G$ is an mtf $(k,n,d)$-graph}  
  			\STATE output $G$
  		\ENDIF
  	\ELSE
  		\FOR{$j$ = set\_index ; $j <$ num.\ max.\ indep.\ sets of order set\_order ; $j$++}
			\STATE Let $S_j$ be max.\ indep.\ set number $j$ of order set\_order of $H$
			\IF{$S_j$ does not contain any \textit{forbidden} vertices}  
				\STATE assign $S_j$ to $n_{num\_assigned+1}$
				\STATE update \textit{forbidden} vertices
				\STATE \verb|Connect_indep_sets|($H$, $d$, set\_order, $j$, num\_assigned + 1)
				\STATE restore \textit{forbidden} vertices
		    \ENDIF
		\ENDFOR
		\IF{set\_order $> 1$}  
			\STATE \verb|Connect_indep_sets|($H$, $d$, set\_order - 1, 0, num\_assigned)
		\ENDIF
  	\ENDIF
  \end{algorithmic}
\end{algorithm}

The \textit{forbidden} vertices in Algorithm~\ref{algo:construct} denote the vertices that have degree $d$. This helps significantly to reduce the number of possible assignments of the neighbours of $v$. 

Note that Brooks' Theorem~\cite{brooks1941colouring} states that for a connected graph $G$ which is not a complete graph or an odd cycle, we have that $\chi(G) \leq \Delta(G)$, so $d \geq k$.
In~\cite{kostochka1982modification} Kostochka proved that for triangle-free graphs $G$ we have that $\chi(G) \leq \frac{2\Delta(G)}{3}+2$, but for $\chi(G) \leq 6$ this does not yield a better lower bound for $\Delta(G)$ than Brooks' theorem. (For $\chi(G) = 7$ it yields $\Delta(G) \geq 8$, which we use to establish the lower bound $41 \leq n(7)$ in Corollary~\ref{cor:bounds_7chrom}).

Furthermore, star graphs are the only mtf graphs with $\delta = 1$, so we can assume that $\delta \geq 2$. Also note that if we only want to generate \textit{$k$-vertex-critical} graphs (i.e.\ $k$-chromatic graphs for which every proper induced subgraph is $(k-1)$-colourable), we know that $\delta \geq k-1$. For vertex-critical graphs we can also assume that no two neighbours of $v$ are connected to the same independent set, since graphs which contain vertices with identical neighbourhoods cannot be vertex-critical.

In principle the number of possible assignments of the neighbours of $v$ to maximal independent sets can be further reduced by investigating $(k-1)$-colourings of the input graphs and determining which independent sets must be selected to lead to a $k$-chromatic graph. 
But as will be seen in Section~\ref{subsect:results}, this is not useful in practice and it was not implemented in the final algorithm.

Our algorithm can be adapted to generate all $(k,n,d)$-graphs directly of just the mtf ones by starting from $(k-1,n-d-1,\leq d)$-graphs (instead of $(k-1,n-d-1,\leq d-1)$-graphs) and connecting the neighbours of $v$ to independent sets -- not just maximal ones. However, it is much more efficient to use our algorithm to first generate all mtf $(k,n)$-graphs and then recursively remove edges to obtain all $(k,n,d)$-graphs for each $d$.

The algorithm can also be adapted to generate $k$-chromatic graphs of girth at least 5. The set of all $(k,n,d)$-graphs of girth at least 5 can be obtained by applying a modified version of Algorithm~\ref{algo:construct} to all $(k - 1,n - d - 1, \leq d)$-graphs of girth at least 5. Now the neighbours of $v$ are connected to independent sets where the vertices are at distance at least 3 from each other -- otherwise a $C_4$ would be formed. Furthermore, the independent sets to which the neighbours of $v$ are connected must be pairwise disjoint. We use this modified algorithm in Section~\ref{subsect:results} to prove that the smallest 5-chromatic graph of girth at least 5 has at least 29 vertices.

In order not to output any isomorphic copies, we use the program \verb|nauty|~\cite{nauty-website, mckay_14} to compute a canonical form of the graphs and remove isomorphisms. This is not a bottleneck as usually only very few $(k,n,d)$-graphs are generated for the values of $k$ and $n$ in this project.


\begin{corollary}
$29 \leq n(6)$.
\end{corollary}

\begin{proof}
Let $G$ be a $(6,\leq 28)$-graph. Because of Brooks' Theorem~\cite{brooks1941colouring}, $G$ must have maximum degree at least 6. So removing a vertex of maximum degree and its neighbours yields a $(5, \leq 28 - 6 - 1 = 21)$-graph, but such graphs do not exist since Jensen and Royle~\cite{jensen1995small} showed that $n(5)=22$.
\end{proof}

\subsection{Testing and results}
\label{subsect:results}


Using the generator for mtf graphs from~\cite{triangleramsey-site,BGSP12} we generated all mtf $(5,n)$-graphs up to $n=24$, which required approximately 13 CPU year on a cluster. By recursively removing edges in all possible ways from these graphs, as long as the graphs stay 5-chromatic, we obtained all triangle-free 5-chromatic graphs of these orders.

The counts of these $(5,n)$-graphs are listed in Table~\ref{table:counts_5chrom}. Previously Jensen and Royle~\cite{jensen1995small} showed that there are at least 80 $(5,22)$-graphs and it follows from our results that this is actually the complete set of all $(5,22)$-graphs. 

Table~\ref{table:counts_5chrom_crit} shows the counts of 5-vertex-critical and 5-critical triangle-free graphs. A \textit{$k$-critical} graph is a $k$-chromatic graph for which every proper subgraph is $(k-1)$-colourable.

For completeness and for comparison, we also included the counts of the smallest triangle-free 4-chromatic graphs in Tables~\ref{table:counts_4chrom}-\ref{table:counts_4chrom_crit}. The graphs from Tables~\ref{table:counts_5chrom_crit} and \ref{table:counts_4chrom_crit} can be downloaded from the \textit{House of Graphs}~\cite{hog} at {\small \url{http://hog.grinvin.org/TrianglefreeKChrom}}.

\begin{table}
\centering
\small
 \renewcommand{\arraystretch}{1.15} 
	\begin{tabular}{c  crrrrrr r}
		 \multirow{2}{*}{$n$} & \multicolumn{8}{c}{Maximum degree} \\
		\cline{2-9}
		  & $<7$ & $7$ & $8$ & $9$ & $10$ & $11$ & $12$ & $>12$ \\
		 \hline
		22 & 0 & 8 & 70 & 2 & 0 & 0 & 0 & 0\\
		23 & 0 & 16 033 & 257 922 & 41 067 & 434 & 1 & 0 & 0\\
		24 & 0 & 3 735 593 & 687 757 507 & 327 307 106 & 11 219 245 & 58 283 & 28 & 0\\		 
	\end{tabular}

\caption{The counts of the smallest $(5,n)$-graphs according to their maximum degree.}
\label{table:counts_5chrom}
\end{table}

\begin{table}
\centering
\small
 \renewcommand{\arraystretch}{1.15} 
	\begin{tabular}{c  rrrr }
		 $n$ & all & \#\,vertex-crit. & \#\,crit. & \#\,mtf\\
		 \hline
		22 & 80 & 80 & 21 & 15 \\
		23 & 315 457 & 154 899 & 4 192 & 2 729 \\
		24 &  1 030 077 762 & 212 827 777 & 625 812 & 369 360 \\		
	\end{tabular}

\caption{The counts of all, vertex-critical, critical and mtf $(5,n)$-graphs, respectively.}
\label{table:counts_5chrom_crit}
\end{table}

\begin{table}
\centering
\small
 \renewcommand{\arraystretch}{1.15} 
	\begin{tabular}{c  crrrrr  r }
		 \multirow{2}{*}{$n$} & \multicolumn{6}{c}{Maximum degree} & \multirow{2}{*}{Total} \\
		\cline{2-7}
		  & $4$ & $5$ & $6$ & $7$ & $8$ & $9$ &  \\
		 \hline
11 & 0 & 1 & 0 & 0 & 0 & 0 &  1 \\
12 & 3 & 18 & 3 & 0 & 0 & 0 & 24 \\
13 & 12 & 814 & 272 & 12 & 0 & 0 &  1 110 \\
14 & 46 & 39 843 & 34 041 & 2 291 & 40 & 0 &  76 261\\
15 & 168 & 1 891 843 & 4 059 278 & 495 873 & 14 099 & 125 & 6 461 386 \\
	\end{tabular}

\caption{The counts of the smallest $(4,n)$-graphs according to their maximum degree.}
\label{table:counts_4chrom}
\end{table}

\begin{table}
\centering
\small
 \renewcommand{\arraystretch}{1.15} 
	\begin{tabular}{c  rrrr }
		 $n$ & all & \#\,vertex-crit. & \#\,crit. & \#\,mtf\\
		 \hline
11 & 1 & 1 & 1 & 1 \\
12 &  24 & 4 & 2 & 5 \\
13 & 1 110 & 31 & 13 & 25 \\
14 &  76 261 & 1 080 & 208 & 151 \\
15 &  6 461 386 & 49 015 & 5 039 & 1 019 \\
	\end{tabular}

\caption{The counts of all, vertex-critical, critical and mtf $(4,n)$-graphs, respectively.}
\label{table:counts_4chrom_crit}
\end{table}


In~\cite{chvatal1970smallest} Chv{\'a}tal determined the smallest 4-regular triangle-free 4-chromatic graph. It has 12 vertices and is also known as the Chv{\'a}tal graph. 
In~\cite{jensen1995graph} Jensen and Toft asked if there exists a 5-regular 5-chromatic triangle-free graph (i.e.\ a special case of Gr\"unbaum's girth problem~\cite{grunbaum1970problem}). We are unable to answer this question, but we can show the following.

\begin{claim}\label{obs:smallest_regular}
The smallest regular triangle-free 5-chromatic graphs have 24 vertices. There are exactly 63 such graphs and all of them are 7-regular.
\end{claim}


We determined that 58 out of the 63 graphs from Claim~\ref{obs:smallest_regular} are 5-vertex-critical, and none of them is 5-critical.  The adjacency list of one of these graphs can be found in the Appendix.
These graphs can also be downloaded from the database of interesting graphs from the \textit{House of Graphs}~\cite{hog} by searching for the keywords ``regular triangle-free 5-chromatic''.
Out of these 63 graphs, 11 have an automorphism group of order 2 and the remaining have a trivial automorphism group. 
 

For 6-chromatic graphs we obtained the following result.

\begin{theorem}
There are no triangle-free 6-chromatic graphs on 31 vertices.
\end{theorem}

\begin{proof}
Let $G$ be an mtf $(6,31)$-graph. The following cases for the maximum degree $d$ of $G$ can occur:

\begin{itemize}
\item Case $d=6$: removing a vertex of degree $d$ and its neighbours yields a $(5,24, \leq 5)$-graph, however it follows from our results from Table~\ref{table:counts_5chrom} that a $(5,24)$-graph has $\Delta \geq 7$.
\item Case $d=7$: removing a vertex of degree $d$ and its neighbours yields a $(5,23, \leq 6)$-graph, however it follows from our results from Table~\ref{table:counts_5chrom} that a $(5,23)$-graph has $\Delta \geq 7$.
\item Case $d=8$: removing a vertex of degree $d$ and its neighbours yields a $(5,22, \leq 7)$-graph. Such graphs do exist but applying Algorithms~\ref{algo:init} and~\ref{algo:construct} to these graphs showed that there are no mtf $(6,31,8)$-graphs.
\item Case $d>8$: removing a vertex of degree $d$ and its neighbours yields a $(5,\leq 21)$-graph, however such graphs do not exist since $n(5)=22$.
\end{itemize}

Finally, note that there is an $(k,n)$-graph if and only if there is an mtf $(k,n)$-graph.
\end{proof}

The proof that no triangle-free 6-chromatic graphs on less than 31 vertices exist is completely analogous and together this gives us the lower bound of $32 \leq n(6)$ in Theorem~\ref{thm:bounds_6chrom}. Also note that a similar reasoning (but now using Kostochka's~\cite{kostochka1982modification} bound $d \geq 8$ for $k=7$) gives a lower bound of 41 for $n(7)$ in Corollary~\ref{cor:bounds_7chrom}.

By applying the maximum degree extension algorithm for higher orders, we were also able to show the following.

\begin{claim}\label{obs:larger_orders}
If a $(6,32)$-graph exists, it must have maximum degree 6. If a $(6,33)$-graph exists, it must have maximum degree 6 or 7.
\end{claim}

Removing a vertex of degree 6 and its neighbourhood from an mtf $(6,32,6)$-graph yields a $(5,25,5)$-graph. So to improve the lower bound of $n(6)$ to 33, it would be sufficient to show that $(5,25,5)$-graphs do not exist. However, it is computationally infeasible to do this using our current methods.

We think it is highly unlikely that the graphs from Claim~\ref{obs:larger_orders} or $(5,25,5)$-graphs exist, since otherwise they would be a counterexample to Reed's $\omega$, $\Delta$, and $\chi$ conjecture~\cite{reed1998omega}. This conjecture says the following.

\begin{conjecture}[Reed~\cite{reed1998omega}]
For any graph $G$, $\chi(G) \leq \left\lceil \frac {\Delta(G)+1+\omega(G)}2 \right\rceil$.
\end{conjecture}

Here $\omega(G)$ denotes the clique number and note that for triangle-free graphs $G$ we have $\omega(G) = 2$. 
Previously, Reed's conjecture was known to hold for triangle-free graphs up to 21 vertices as a consequence of Jensen en Royle's~\cite{jensen1995small} results.

\begin{claim}
Reed's conjecture holds for triangle-free graphs up to at least 24 vertices and for triangle-free graphs $G$ with $\chi(G) \geq 6$ up to at least 31 vertices.
\end{claim}

\begin{proof}
It follows from Brooks' Theorem that Reed's conjecture holds for triangle-free graphs $G$ with $\chi(G) \leq 4$. 
It follows from our results from Table~\ref{table:counts_5chrom} that $(5,22)$-, $(5,23)$- and $(5,24)$-graphs have $\Delta \geq 7$ and from the lower bound for $n(6)$ that Reed's conjecture holds for triangle-free graphs up to at least 24 vertices.
The second statement follows immediately from our bound $32 \leq n(6)$.
\end{proof}


\subsubsection*{$k$-chromatic graphs of higher girth}

Recall that Royle~\cite{royle2015} showed that the smallest 4-chromatic graphs of girth at least 5 have 21 vertices and that there are exactly 18 such graphs.
Using the generator \verb|geng|~\cite{nauty-website, mckay_14} we determined that there are exactly 1588 4-chromatic graphs of girth at least 5 on 22 vertices. Exactly 625 of those graphs are 4-vertex-critical and 319 are 4-critical. Figure~\ref{fig:22v_maxceg7} shows the smallest 4-critical graph of girth 5 and maximum degree 7. It has 22 vertices and an automorphism group of order 7. (Note that the 18 4-chromatic graphs of girth at least 5 on 21 vertices all have maximum degree less than 7).

\begin{figure}[h!t]
	\centering
	\includegraphics[width=0.6\textwidth]{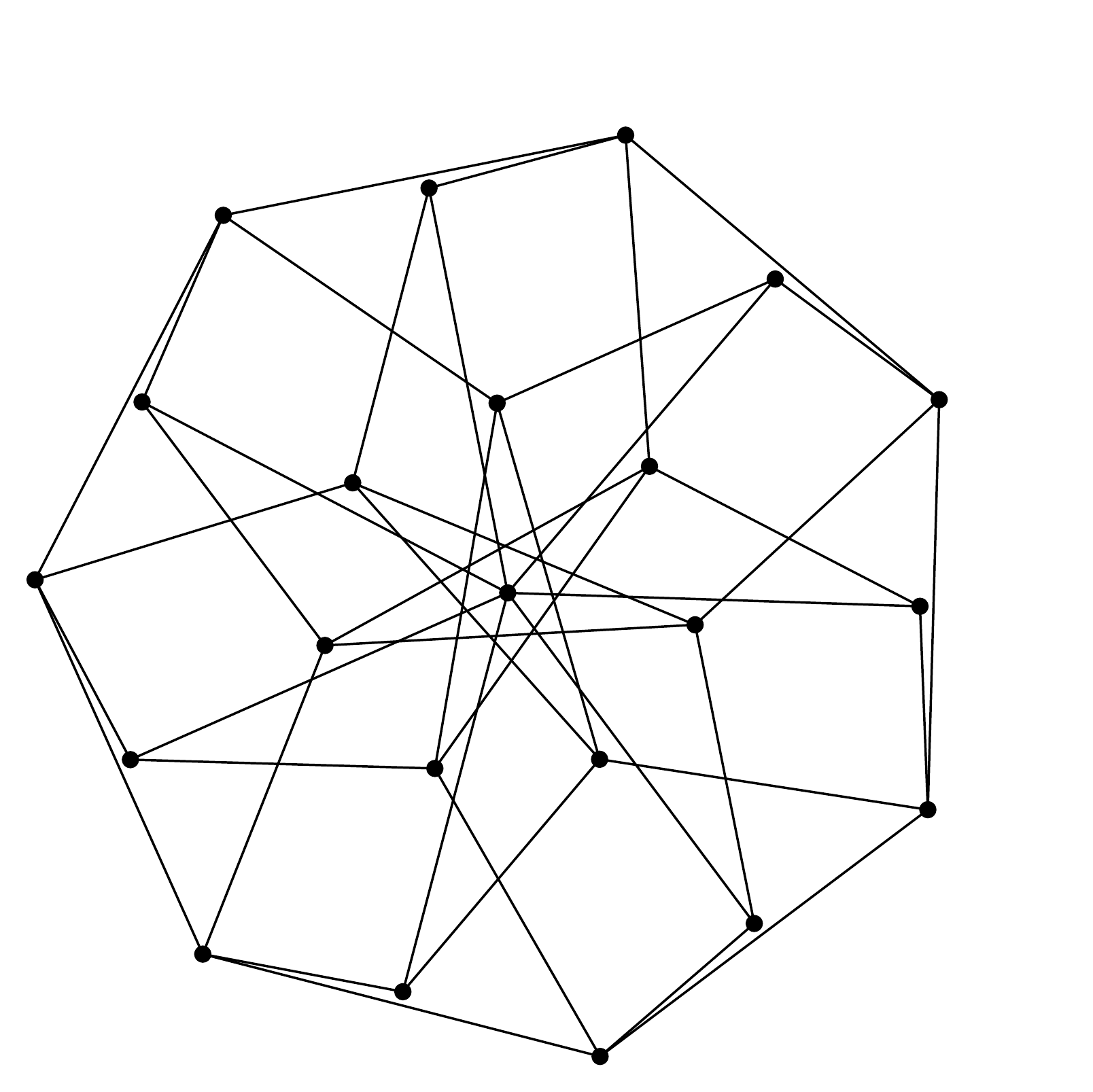}
	\caption{The smallest 4-critical graph of girth 5 and maximum degree 7.}
	\label{fig:22v_maxceg7}
\end{figure}

By adapting \verb|geng| to generate graphs with girth at least 6, we were able to show the following.

\begin{claim}
The smallest 4-chromatic graph of girth at least 6 has at least 25 vertices.
\end{claim}

For 5-chromatic graphs of girth at least 5 we can show the following.

\begin{theorem}
The smallest 5-chromatic graph of girth at least 5 has at least 29 vertices.
\end{theorem}
\begin{proof}
Let $G$ be a 5-chromatic graph of girth at least 5 of order $n \leq 28$. The following cases for $n$ and the maximum degree $d$ of $G$ can occur:

\begin{itemize}
\item Case $n \leq 26$: removing a vertex of maximum degree yields a $(4,\leq 26 - 5 - 1 = 20)$-graph of girth at least 5, however such graphs do not exist since the smallest 4-chromatic graphs of girth at least 5 have 21 vertices.

\item Case $n = 27$: 
\begin{itemize}
\item Case $d=5$:  removing a vertex of degree $d$ and its neighbours yields a $(4,21, \leq 5)$-graph of girth at least 5. Such graphs do exist but applying the modified version of Algorithm~\ref{algo:construct} for girth 5 to these graphs showed that there are no $(5,27,5)$-graphs of girth at least 5.

\item Case $d>5$:  removing a vertex of degree $d$ and its neighbours yields a $(4,\leq 20)$-graph of girth at least 5, but such graphs do not exist.
\end{itemize}

\item Case $n = 28$: 
\begin{itemize}
\item Case $d=5$:  removing a vertex of degree $d$ and its neighbours yields a $(4,22, \leq 5)$-graph of girth at least 5. 
Such graphs do exist but applying the modified version of Algorithm~\ref{algo:construct} for girth 5 to these graphs showed that there are no $(5,28,5)$-graphs of girth at least 5. 

\item Case $d=6$:  removing a vertex of degree $d$ and its neighbours yields a $(4,21, \leq 6)$-graph of girth at least 5. Such graphs do exist but applying the modified version of Algorithm~\ref{algo:construct} for girth 5 to these graphs showed that there are no $(5,28,6)$-graphs of girth at least 5.

\item Case $d>6$:  removing a vertex of degree $d$ and its neighbours yields a $(4,\leq 20)$-graph of girth at least 5, but such graphs do not exist.
\end{itemize}
\end{itemize}
\end{proof}

By applying the maximum degree extension algorithm for order 29, we were also able to show the following.

\begin{claim}
If a 5-chromatic graph of girth at least 5 on 29 vertices exists, it must have maximum degree 5. 
\end{claim}


\subsubsection*{Correctness tests}

The mtf and maximum degree extension method are complementary. For example: it is computationally infeasible to determine all $(5,24,\leq 7)$-graphs using the latter method. Furthermore we also used the mtf method for correctness tests of our implementations.
More precisely, we used both methods to independently determine the complete sets of all mtf $(4,14)$-, $(4,15)$-, $(4,16)$-  and $(5,22)$-graphs. These sets contain 151, 1019, 8357 and 15 mtf graphs, respectively, and in each case the results of both programs were in complete agreement. 

We also used an independent program to verify that the graphs generated by our programs are indeed $k$-chromatic. Furthermore several of the routines (e.g.\ to compute independent sets) were already used an tested before in our work to generate triangle Ramsey graphs from~\cite{staszek13}. 
Finally, Jensen and Royle listed counts of 4-vertex-critical and 4-critical triangle-free graphs for small orders and small maximum degree in~\cite{jensen1995small} and these counts are in complete agreement with our results from Table~\ref{table:counts_4chrom_crit}.


\section{Improving the upper bound}
\label{section:upper_bound}

In this section we describe how we improved the upper bound for $n(6)$. Recall that the smallest triangle-free 6-chromatic graph known so far was the graph on 44 vertices found by Droogendijk~\cite{droogendijk2015}. 

In Algorithm~\ref{algo:upper_bound} we present the pseudocode of the heuristic method we used to search for smaller triangle-free 6-chromatic graphs. Basically, we start from a set of mtf $(k,n)$-graphs and then recursively remove and add edges in all possible ways until no new $(k,n)$-graphs are found or until non-$k$-vertex-critical graphs are found.  We then remove non-critical vertices from the non-vertex-critical graphs to obtain $k$-chromatic graphs with $n-1$ vertices and repeat Algorithm~\ref{algo:upper_bound} on these $(k,n-1)$-graphs. 
Note that if a graph $G$ is not vertex-critical, the mtf graphs obtained by adding edges to $G$ will be non-vertex-critical as well.

\begingroup
\renewcommand{\baselinestretch}{1.05} 
\begin{algorithm}[h]
\caption{Algorithm to search for non-vertex-critical $(k,n)$-graphs}
\label{algo:upper_bound}
  \begin{algorithmic}[1]
  \STATE given a set of mtf $(k,n)$-graphs $\mathcal{M}$
  \STATE let $\mathcal{N}$ be an empty set \ // will contain a set of non-vertex-critical graphs
  \REPEAT
  			\FOR{every $k$-critical subgraph of the graphs in $\mathcal{M}$}
  				\STATE add edges to $G'$ in all possible ways (without creating any triangles)
  				\FOR{every mtf graph $G'_{mtf}$ obtained from $G'$}
  					\STATE add $G'_{mtf}$ to $\mathcal{M}$
  					\IF{$G'_{mtf}$ is not $k$-vertex-critical}
  						\STATE add $G'_{mtf}$ to $\mathcal{N}$
  					\ENDIF
  				\ENDFOR
  			\ENDFOR
  \UNTIL{no new graphs were added to $\mathcal{M}$ or until $\mathcal{N}$ contains ``enough" graphs}
  \end{algorithmic}
\end{algorithm}
\endgroup

For the first iteration, we started Algorithm~\ref{algo:upper_bound} from the $(6,44)$-graph found by Droogendijk as well as from the $(6,45)$-graphs obtained by applying the Mycielski construction to the 80 $(5,22)$-graphs. By recursively applying this algorithm, we obtained the following result.

\begin{claim} \label{obs:upper_bound}
The smallest triangle-free 6-chromatic graphs have at most 40 vertices. There are at least 750\,000 $(6,40)$-graphs.
\end{claim}

The maximum degree of those more than 750\,000 $(6,40)$-graphs from Claim~\ref{obs:upper_bound} ranges between 11 and 14. All of these graphs have an automorphism group of order 1 or 2 and about 5\,000 of these graphs are mtf.


The adjacency list of one of these $(6,40)$-graphs can be found in the Appendix and more can be downloaded from the database of interesting graphs from the \textit{House of Graphs}~\cite{hog} by searching for the keyword ``(6,40)-graph''.

We also applied our maximum degree extension algorithm from Section~\ref{subsect:maxdeg} on samples of our sets of 5-chromatic graphs (it was computationally infeasible to do this on the complete sets), but this did not yield any $(6,<40)$-graphs. To conclude, we believe that the actual value of $n(6)$ is much closer to 40 than to the current lower bound of 32.


\subsection*{Acknowledgements}

We would like to thank Staszek Radziszowski and Vera Weil for useful suggestions.
Most of the computations were carried out using the Stevin Supercomputer Infrastructure at Ghent University.


\bibliographystyle{plain}
\bibliography{references}

\section*{Appendix}


\subsection*{Adjacency list of a smallest regular triangle-free 5-chromatic graph}

Below is the adjacency list one of the 7-regular triangle-free 5-chromatic graphs from Claim~\ref{obs:smallest_regular}. 



\begin{center}
\renewcommand{\arraystretch}{0.9}
\small
\begin{tabular}{r r r r r r r r}
0: & 7 & 11 & 12 & 15 & 17 & 19 & 22\\
1: & 2 & 4 & 7 & 9 & 14 & 19 & 21\\
2: & 1 & 3 & 10 & 11 & 13 & 20 & 22\\
3: & 2 & 5 & 7 & 9 & 14 & 19 & 23\\
4: & 1 & 5 & 8 & 16 & 17 & 18 & 23\\
5: & 3 & 4 & 6 & 10 & 15 & 20 & 21\\
6: & 5 & 8 & 9 & 11 & 13 & 17 & 19\\
7: & 0 & 1 & 3 & 8 & 13 & 16 & 20\\
8: & 4 & 6 & 7 & 10 & 14 & 15 & 21\\
9: & 1 & 3 & 6 & 12 & 15 & 18 & 22\\
10: & 2 & 5 & 8 & 12 & 17 & 18 & 23\\
11: & 0 & 2 & 6 & 14 & 16 & 18 & 23\\
12: & 0 & 9 & 10 & 13 & 16 & 20 & 21\\
13: & 2 & 6 & 7 & 12 & 14 & 15 & 18\\
14: & 1 & 3 & 8 & 11 & 13 & 17 & 22\\
15: & 0 & 5 & 8 & 9 & 13 & 16 & 23\\
16: & 4 & 7 & 11 & 12 & 15 & 19 & 22\\
17: & 0 & 4 & 6 & 10 & 14 & 20 & 21\\
18: & 4 & 9 & 10 & 11 & 13 & 19 & 21\\
19: & 0 & 1 & 3 & 6 & 16 & 18 & 20\\
20: & 2 & 5 & 7 & 12 & 17 & 19 & 23\\
21: & 1 & 5 & 8 & 12 & 17 & 18 & 22\\
22: & 0 & 2 & 9 & 14 & 16 & 21 & 23\\
23: & 3 & 4 & 10 & 11 & 15 & 20 & 22\\
\end{tabular}
\end{center}

\newpage

\subsection*{Adjacency list of a triangle-free 6-chromatic graph on 40 vertices}

Below is the adjacency list of of one of the triangle-free 6-chromatic graphs on 40 vertices from Claim~\ref{obs:upper_bound}.





\begin{center}
\renewcommand{\arraystretch}{0.9}
\small
\begin{tabular}{r r r r r r r r r r r r r}
0: & 10 & 24 & 25 & 26 & 28 & 29 & 33 & 37 & 38 & 39 &  &  \\
1: & 7 & 11 & 24 & 25 & 27 & 28 & 30 & 37 & 38 & 39 &  &  \\
2: & 10 & 12 & 23 & 26 & 27 & 29 & 34 & 37 & 38 & 39 &  &  \\
3: & 10 & 13 & 23 & 29 & 31 & 34 & 35 & 36 & 38 & 39 &  &  \\
4: & 12 & 13 & 21 & 23 & 27 & 29 & 32 & 34 & 36 & 38 & 39 &  \\
5: & 11 & 13 & 20 & 21 & 23 & 27 & 35 & 36 & 38 & 39 &  &  \\
6: & 11 & 13 & 20 & 22 & 24 & 25 & 30 & 35 & 38 & 39 &  &  \\
7: & 1 & 9 & 13 & 18 & 22 & 29 & 31 & 33 & 36 &  &  &  \\
8: & 11 & 12 & 19 & 27 & 28 & 29 & 32 & 37 & 39 &  &  &  \\
9: & 7 & 16 & 17 & 20 & 23 & 24 & 27 & 30 & 35 & 38 &  &  \\
10: & 0 & 2 & 3 & 11 & 14 & 18 & 19 & 21 & 30 & 32 &  &  \\
11: & 1 & 5 & 6 & 8 & 10 & 15 & 16 & 17 & 26 & 31 & 34 &  \\
12: & 2 & 4 & 8 & 17 & 20 & 22 & 24 & 25 & 30 & 35 &  &  \\
13: & 3 & 4 & 5 & 6 & 7 & 14 & 15 & 16 & 17 & 26 & 37 &  \\
14: & 10 & 13 & 22 & 24 & 25 & 29 & 31 & 33 & 35 & 38 & 39 &  \\
15: & 11 & 13 & 20 & 21 & 23 & 27 & 30 & 32 & 38 & 39 &  &  \\
16: & 9 & 11 & 13 & 18 & 19 & 22 & 25 & 28 & 29 & 33 & 39 &  \\
17: & 9 & 11 & 12 & 13 & 18 & 19 & 21 & 29 & 32 & 36 & 39 &  \\
18: & 7 & 10 & 16 & 17 & 23 & 24 & 26 & 27 & 34 & 37 & 38 &  \\
19: & 8 & 10 & 16 & 17 & 20 & 23 & 24 & 31 & 34 & 35 & 38 &  \\
20: & 5 & 6 & 9 & 12 & 15 & 19 & 26 & 28 & 29 & 33 & 37 &  \\
21: & 4 & 5 & 10 & 15 & 17 & 24 & 25 & 26 & 28 & 33 & 37 &  \\
22: & 6 & 7 & 12 & 14 & 16 & 23 & 26 & 27 & 32 & 34 & 37 &  \\
23: & 2 & 3 & 4 & 5 & 9 & 15 & 18 & 19 & 22 & 25 & 28 & 33 \\
24: & 0 & 1 & 6 & 9 & 12 & 14 & 18 & 19 & 21 & 32 & 36 &  \\
25: & 0 & 1 & 6 & 12 & 14 & 16 & 21 & 23 & 32 & 34 & 36 &  \\
26: & 0 & 2 & 11 & 13 & 18 & 20 & 21 & 22 & 30 & 35 &  &  \\
27: & 1 & 2 & 4 & 5 & 8 & 9 & 15 & 18 & 22 & 31 & 33 &  \\
28: & 0 & 1 & 8 & 16 & 20 & 21 & 23 & 31 & 34 & 35 & 36 &  \\
29: & 0 & 2 & 3 & 4 & 7 & 8 & 14 & 16 & 17 & 20 & 30 &  \\
30: & 1 & 6 & 9 & 10 & 12 & 15 & 26 & 29 & 31 & 33 & 34 & 36 \\
31: & 3 & 7 & 11 & 14 & 19 & 27 & 28 & 30 & 32 & 37 &  &  \\
32: & 4 & 8 & 10 & 15 & 17 & 22 & 24 & 25 & 31 & 33 & 35 &  \\
33: & 0 & 7 & 14 & 16 & 20 & 21 & 23 & 27 & 30 & 32 &  &  \\
34: & 2 & 3 & 4 & 11 & 18 & 19 & 22 & 25 & 28 & 30 &  &  \\
35: & 3 & 5 & 6 & 9 & 12 & 14 & 19 & 26 & 28 & 32 & 37 &  \\
36: & 3 & 4 & 5 & 7 & 17 & 24 & 25 & 28 & 30 & 37 &  &  \\
37: & 0 & 1 & 2 & 8 & 13 & 18 & 20 & 21 & 22 & 31 & 35 & 36 \\
38: & 0 & 1 & 2 & 3 & 4 & 5 & 6 & 9 & 14 & 15 & 18 & 19 \\
39: & 0 & 1 & 2 & 3 & 4 & 5 & 6 & 8 & 14 & 15 & 16 & 17 \\
\end{tabular}
\end{center}

\end{document}